\documentclass[12pt, reqno]{amsart}

\usepackage{amsmath, amsthm, amssymb, stmaryrd}
\usepackage{hyperref}
\usepackage{enumerate}
\usepackage{url}

\usepackage{color}

\usepackage{graphicx}
\usepackage{multirow, array}

\usepackage[centering, margin={1in, 0.5in}, includeheadfoot]{geometry}

\usepackage{fancyvrb}

\newtheorem{thm}{Theorem}[section]
\newtheorem{lemma}[thm]{Lemma}

\newtheorem{cor}[thm]{Corollary}
\newtheorem{conj}[thm]{Conjecture}

\theoremstyle{definition}

\newtheorem{example}[thm]{Example}

\theoremstyle{remark}

\numberwithin{equation}{section}

\def\lam{{\lambda}}

\def\eps{\varepsilon}

\def\le{\leqslant} \def\ge{\geqslant}

\def \bN {\mathbb N}

\def \bR {\mathbb R}
\def \bZ {\mathbb Z}

\def \bf {\mathbf f}

\begin{document}
\title[On Fibonacci partitions]{On Fibonacci partitions}
\subjclass[2010]{11B39 (primary); 05A16, 05A17 (secondary)}
\keywords{Fibonacci numbers, partitions}
\author{Sam Chow \and Tom Slattery}
\address{Mathematics Institute, Zeeman Building, University of Warwick, Coventry CV4 7AL, United Kingdom}
\email{Sam.Chow@warwick.ac.uk}
\address{Mathematics Institute, Zeeman Building, University of Warwick, Coventry CV4 7AL, United Kingdom}
\email{t.slattery@warwick.ac.uk}

\maketitle

{\centering\footnotesize \em{To Carl Pomerance, a legend of number theory} \par}

\begin{abstract} We prove an exact formula for OEIS A000119, which counts partitions into distinct Fibonacci numbers. We also establish an exact formula for its mean value, and determine the asymptotic behaviour.
\end{abstract}

\section{Introduction}

For $n \in \bZ_{\ge 0}$, let $R(n)$ be the number of solutions to
\[
x_1 + \cdots + x_s = n,
\]
where $s \in \bZ_{\ge 0}$, and $x_1 < x_2 < \cdots < x_s$ are Fibonacci numbers. Note that $R(0) = 1$. We call $R$ the \emph{Fibonacci partition function}, as it counts partitions into distinct Fibonacci numbers. It has existed since the very first volume of the Fibonacci Quarterly in 1963, see \cite{HB1963}, and its values comprise the sequence OEIS A000119. In 1968, Leonard Carlitz \cite[Theorem 2]{Car1968} showed that 
\begin{equation} \label{Carlitz}
R(F_m) = \lfloor m/2 \rfloor \qquad (m = 2,3,\ldots),
\end{equation}
where $F_1 = F_2 = 1$, $F_3 = 2$, and if $m \in \bN$ then $F_m$ denotes the $m^\mathrm{th}$ Fibonacci number. Many authors have investigated Fibonacci partitions, and the topic has received attention over several decades \cite{Ard2004, Berstel, Car1968, HB1963, Kla1968, Rob1996, Stockmeyer, Wei2016}. In general $R(n)$ behaves erratically. 

\begin{figure}[!htb] \label{OriginalFigure}
  \centering
\includegraphics[width=8cm]{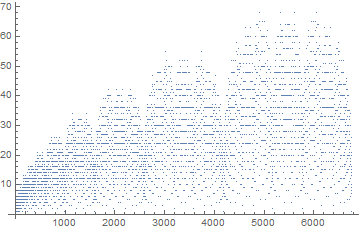}
  \caption{$R(n)$ against $n$ for $n=0,1,\ldots,6765$}
\end{figure}

Our first result is an exact formula for $R(n)$. Recall Zeckendorf's theorem \cite{Knuth, Lekkerkerker, Zeckendorf}, which asserts that each positive integer has a unique representation as a sum of non-consecutive Fibonacci numbers, called the \emph{Zeckendorf expansion}. 

\begin{thm} \label{thm1}
Let
\[
H = F_{m_0} + F_{m_1} + \cdots + F_{m_k}
\]
be the Zeckendorf expansion of $H \in \bN$, where
\[
m_{i-1} - m_i \ge 2 \qquad (1 \le i \le k), \qquad m_k \ge 2.
\]
Write
\[
x_\ell = F_{m_\ell} + \cdots + F_{m_k} \quad (0 \le \ell \le k+1)
\]
and
\[
t_i = \Big \lfloor \frac{m_{i-1} - m_i + 2}2 \Big \rfloor,
\quad
\eps_i = 2t_i - 1 - m_{i-1} + m_i \qquad (1 \le i \le k).
\]
Finally, let
\[
a_0 = 1, \qquad a_1 = t_1, \qquad
a_{\ell +1} = t_{\ell +1}a_{\ell} - \eps_\ell a_{\ell-1} \quad (1 \le \ell \le k-1).
\]
Then
\begin{equation} \label{ExactFormula}
R(H) =
\begin{cases}
a_k \lfloor m_k/2  \rfloor - \eps_k a_{k-1}, &\text{if } k \ge 1 \\
\lfloor m_0/2  \rfloor, &\text{if } k = 0.
\end{cases}
\end{equation}
\end{thm}

\noindent Throughout, we adopt the standard convention that empty sums are 0, so $x_{k+1} = 0$ above.

Carlitz \cite{Car1968} had a recursive formula, but on attempting to produce a non-recursive formula found that ``the general case is very complicated''. Robbins \cite{Rob1996} had a simpler recursive formula, leading to an algorithm used to produce some initial values of $R(H)$, but also did not write down a non-recursive formula. Weinstein \cite{Wei2016} obtained a non-recursive expression, albeit a complicated one. The nicest formula that we could find in the literature is that of Berstel \cite[Proposition 3.1]{Berstel} which, being a product of $2 \times 2$ matrices, is quite similar to Theorem \ref{thm1}. Our formula \eqref{ExactFormula} is extremely efficient in practice. For example, it can compute $R(10^{100})$ in less than one second on a standard laptop computer. \emph{Mathematica} \cite{Wolfram} code for this is provided in Appendix \ref{Rcode}. Theorem \ref{thm1} follows readily from Robbins's recursion, so it is not our main result by any means.

\bigskip

We also study the mean value
\[
M(H) := H^{-1} \sum_{n = 0}^H R(n) \qquad (H \in \bN),
\]
or equivalently the summatory function
\[
A(H) := \sum_{n = 0}^H R(n) \qquad (H \in \bZ).
\]
For $t \in \bN$, let
\begin{equation} \label{fdef}
f(t) = 1 + \frac{2(4^{t-1} -1)}3.
\end{equation}
We establish the following exact formula for $A(H)$.

\begin{thm} \label{SuperDuperRecursion}
Let $H \in \bN$, and let the values of the $x_\ell$, $t_i$, $\eps_i$ and $a_\ell$ be as in Theorem \ref{thm1}.
Then for $\ell = 1,2,\ldots,k$ we have
\begin{equation} \label{GeneralRecursion}
A(H) = a_\ell A(x_\ell) - \eps_\ell a_{\ell - 1} A(x_{\ell+1})
+ 
\sum_{i \le \ell} a_{i-1} f(t_i) 2^{m_{i-1} - 2t_i}.
\end{equation}
In particular
\begin{equation} \label{FinalFormula}
A(H) =
\begin{cases}
a_k \Bigg \lfloor \frac{2^{m_k}}6 + \frac{m_k + 1}2 \Bigg \rfloor - \eps_k a_{k-1} + \displaystyle
\sum_{i \le k} a_{i-1} f(t_i) 2^{m_{i-1} - 2t_i}, &\text{if } k \ge 1 \\
\Bigg \lfloor \frac{2^{m_0}}6 + \frac{m_0 + 1}2 \Bigg \rfloor, &\text{if } k = 0.
\end{cases}
\end{equation}
\end{thm}

This enables us to understand the asymptotic behaviour of $A(H)$ and $M(H)$. Put
\[
\varphi = \frac{1+\sqrt 5}2, \qquad \lam = \frac{\log 2}{\log \varphi} \approx 1.44,
\]
and define
\[
c_1 = \liminf_{H \to \infty} \frac{A(H)} {H^{\lam}},
\qquad c_2 = \limsup_{H \to \infty} \frac{A(H)} {H^{\lam}}.
\]
We now present our main result.

\begin{thm} [Main Theorem] \label{AsymptoticBehaviour} We have
\[
c_1 = 0.52534\ldots, \qquad c_2 = 0.54338\ldots,
\]
and more precisely
\begin{equation} \label{close}
0.525347 < c_1 < 0.525349,
\qquad
0.5433878 < c_2 < 0.5433893.
\end{equation}
\end{thm}

\noindent It follows that $A(H) \asymp H^\lam$. Subject to hardware constraints, our method computes $c_1$ and $c_2$ to arbitrary precision. 

\begin{figure}[!htb]
  \centering
\includegraphics[width=8cm]{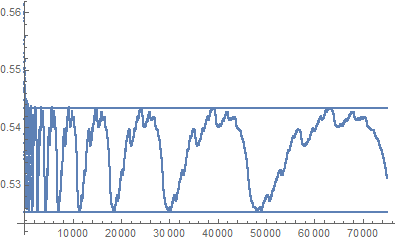}
  \caption{$A(H) / H^{\lam}$ against $H$ for $H=0,1,\ldots,75025$, with horizontal lines at $0.525348$ and $0.543388$.}
\end{figure}

\bigskip

The Fibonacci partition function behaves very differently to the usual partition function $p(n)$, for which there is a nice asymptotic formula
\[
p(n) \sim \frac{1}{4n \sqrt 3} \exp ( \pi \sqrt {2n/3})
\]
going back to Hardy and Ramanujan \cite{HR1918}; see also \cite[\S 5]{Andrews}. Our work shows that even the mean value $M(H)$ of the Fibonacci partition function does not have a `nice' asymptotic formula, however we are able to describe the asymptotic behaviour fairly well.

The logarithmic average of $R(n) n^{1-\lam}$, namely
\[
B(H) := (\log H)^{-1} \sum_{n \le H} \frac{R(n)}{n^\lam} \qquad (H \ge 2),
\]
might be better behaved. Breaking into ranges $I_m = (F_m, F_{m+1}]$, wherein
\[
\frac{A(F_{m+1}) - A(F_m)}{F_{m+1}^\lam}
\le\sum_{n \in I_m} \frac{R(n)}{n^\lam} \le \frac{A(F_{m+1}) - A(F_m)}{F_m^\lam},
\]
it follows from Theorem \ref{AsymptoticBehaviour} that 
\[
B(H) \asymp 1.
\]
Though $B(H)$ is not decreasing, it does exhibit a clear downward trend. 

\begin{conj} There exists $B > 0$ such that
\[
B(H) \to B \qquad (H \to \infty).
\]
\end{conj}

We also invite the enthusiastic reader to consider:
\begin{enumerate}
\item Higher moments of the Fibonacci partition function
\item Lucas partitions \cite{CSH1972, Kla1966}
\item Partitions into distinct terms of a sequence $(\lfloor \tau^m \rfloor)_{m=1}^\infty$, where $\tau \in (1,2) $ is fixed
\item Partitions into distinct terms of a Piatetski-Shapiro sequence $(\lfloor m^\tau \rfloor)_{m=1}^\infty$, where $\tau > 1$ is fixed, cf. for polynomials \cite{DR2018, Gafni}
\item Partitions into distinct Piatetski--Shapiro primes, cf. \cite{Kumchev, Vaughan}.
\end{enumerate}

\subsection*{Methods} 

We deduce Theorem \ref{thm1} by iterating Robbins's recursion \cite[Theorem 4]{Rob1996}. For Theorem \ref{SuperDuperRecursion}, we begin with the observation that $A(H)$ counts sets of distinct Fibonacci numbers whose sum is at most $H$. This enables us to prove a combinatorial recursion analogous to that of Robbins. By systematic applications of our recursion, we prove an exact formula for $A(H)$ in terms of the Zeckendorf expansion of $H$. Finally, for $m \in \bN$ large, we subdivide $[F_m, F_{m+1}) \cap \bZ$ into many discrete subintervals, according to the initial Zeckendorf digits. By estimating $A(H)$ at the endpoints of these subintervals, we are able to compute $c_1$ and $c_2$ to arbitrary precision, subject to hardware constraints. We used the software \emph{Mathematica} \cite{Wolfram} to perform the calculations, leading to Theorem \ref{AsymptoticBehaviour}. 

\subsection*{Notation} As usual, empty sums are 0. We adopt the following standard asymptotic notations: if $f,g: \bN \to \bR_{> 0}$, we write
\begin{align*}
f(m) \sim g(m)
\qquad
\text{if}
\qquad
\lim_{m \to \infty} \frac{f(m)}{g(m)} = 1,\\
f(m) = o(g(m)) \qquad \text{if} \qquad 
\lim_{m \to \infty} \frac{f(m)}{g(m)} = 0,
\end{align*}
and 
\[
f(m) \asymp g(m)
\qquad
\text{if}
\qquad
0 < \liminf_{m \to \infty} \frac{f(m)}{g(m)} \le \limsup_{m \to \infty} \frac{f(m)}{g(m)} < \infty.
\]
In words, the first notion is that $f$ is asymptotic to $g$, the second notion is that $f$ has a smaller asymptotic order of magnitude than $g$, and the third notion is that $f$ and $g$ have the same asymptotic order of magnitude.

\subsection*{Organisation} We prove Theorems \ref{thm1}, \ref{SuperDuperRecursion} and \ref{AsymptoticBehaviour} in Sections \ref{S2}, \ref{S3} and \ref{S4}, respectively. The appendices contain the code that we used for the computations.

\subsection*{Funding}

SC was supported by EPSRC Fellowship Grant EP/S00226X/2. TS was supported by a URSS bursary from the University of Warwick. 

\section{An exact formula for Fibonacci partitions} \label{S2}

In this section, we prove Theorem \ref{thm1}. With the notation of Theorem \ref{thm1}, Robbins \cite[Theorem 4]{Rob1996} established the following recursion.

\begin{lemma} [Robbins] \label{Robbins}
If $H \ge 2$ and $k \ge 1$ then
\[
R(H) = t_1 R(x_1) - \eps_1 R(x_2).
\]
\end{lemma}

We induct on $\ell$ to show that if $\ell = 1,2,\ldots,k$ then
\begin{equation} \label{SuperRobbins}
R(H) = a_\ell R(x_\ell) - \eps_\ell a_{\ell - 1} R(x_{\ell+1}).
\end{equation}
The base case $\ell = 1$ is Lemma \ref{Robbins}. Now suppose $1 \le \ell \le k-1$, and that \eqref{SuperRobbins} holds. Then
\begin{align*}
R(H) &= a_\ell R(x_\ell) - \eps_\ell a_{\ell - 1} R(x_{\ell+1}) \\
&= a_\ell (t_{\ell+1} R(x_{\ell+1}) - \eps_{\ell + 1} R(x_{\ell + 2}))
- \eps_\ell a_{\ell - 1} R(x_{\ell+1}) \\
&= a_{\ell+1} R(x_{\ell+1}) - \eps_{\ell + 1} a_\ell R(x_{\ell + 2}),
\end{align*}
which is \eqref{SuperRobbins} with $\ell + 1$ in place of $\ell$. Thus, we have established \eqref{SuperRobbins} by induction.

For $k \ge 1$, applying \eqref{SuperRobbins} with $\ell = k$, and then applying \eqref{Carlitz} with $m = m_k$, gives
\[
R(H) = a_k R(m_k) - \eps_k a_{k-1} = a_k \lfloor m_k / 2 \rfloor  - \eps_k a_{k-1}.
\]
The $k = 0$ case of \eqref{ExactFormula} is \eqref{Carlitz}, which was already established by Carlitz \cite[Theorem 2]{Car1968}. This completes the proof of Theorem \ref{thm1}.

\section{The summatory function} \label{S3}

In this section, we prove Theorem \ref{SuperDuperRecursion}.

\subsection{A combinatorial recursion}

Recall that the Fibonacci sequence enjoys the recursive relation $F_{m+1} = F_m + F_{m-1}$. Our starting point is the following recursion for the summatory function $A(H)$.

\begin{lemma} \label{KeyRecursion} If $m \in \bZ_{\ge 3}$ and $F_m  \le H < F_{m+1}$ then
\[
A(H) = A(H-F_m) + A(H - F_{m-1}) - A(H - 2F_{m-1}) + 2^{m-3}.
\]
\end{lemma}

\begin{proof} Observe that $A(H)$ counts tuples 
$
(x_1,\ldots,x_s)
$
of Fibonacci numbers such that 
\[
s \ge 0, \qquad
x_1 < \cdots < x_s, \qquad x_1 + \cdots + x_s \le H.
\]
Note that $x_1,\ldots,x_s \in \{ F_2, \ldots, F_m \}$, since $F_1= F_2$. There are $A(H-F_m)$ such tuples for which $x_s = F_m$, since $H - F_m < F_m$. 

If $x_s = F_{m-1}$, then we have
\[
x_1 + \cdots + x_{s-1} \le H - F_{m-1} < F_m < 2F_{m-1}.
\]
There would be $A(H-F_{m-1})$ solutions to this if $x_{s-1}$ were allowed to equal $F_{m-1}$, but since $x_{s-1} < x_s$ this is forbidden, and we need to subtract $A(H- 2F_{m-1})$. Thus, there are 
\[
A(H - F_{m-1}) - A(H - 2F_{m-1})
\]
valid tuples for which $x_s = F_{m-1}$.

Finally, if $x_1 < x_2 < \ldots < x_s \le F_{m-2}$ are Fibonacci numbers, then we always have
\[
x_1 + \cdots + x_s \le F_2 + \cdots + F_{m-2} < F_m \le H,
\]
owing to the well-known identity 
\begin{equation*}
F_1 + F_2 + \cdots + F_{n-2} = F_n - 1 \qquad (n \in \bN),
\end{equation*}
the proof of which is a straightforward exercise in mathematical induction. As there are $2^{m-3}$ subsets of $ \{ F_2, \ldots, F_{m-2} \}$, there are $2^{m-3}$ valid tuples for which $x_s \le F_{m-2}$.

Summing the contributions from the three cases completes the proof of the lemma.
\end{proof}

Next, we provide a simple argument to show that
\begin{equation} \label{OrderOfMagnitude}
A(H) \asymp 	H^{\lam},
\end{equation}
recalling our notational convention that this describes the asymptotic order of magnitude as $H \to +\infty$. Let $m \ge 4$ be an integer. If $m$ is odd then, by Lemma \ref{KeyRecursion}, we have
\begin{align*}
A(F_m) &= 2^{m-3} + 1 + A(F_{m-2}) = 
\ldots \\
& = (2^{m-3} + 1) + (2^{m-5} + 1) + \cdots + (2^2 + 1) + A(F_3) \\
& = (1 + 2^2 + \cdots + 2^{m-3}) + (m+1)/2 \\
&= \frac{2^{m-1}-1}3 + \frac{m+1}2 = \Bigg \lfloor\frac{2^m}6 + \frac{m+1}2 \Bigg \rfloor.
\end{align*}
Similarly, when $m \ge 4$ is even we reach the same eventual conclusion, and we can check directly that it also holds when $m=2,3$. Thus, we have
\begin{equation} \label{FibonacciExact}
A(F_m) = \Bigg \lfloor \frac{2^m}6 + \frac{m+1}2 \Bigg \rfloor  \qquad (m \ge 2)
\end{equation}
and
\begin{equation} \label{FibonacciAsymptotic}
A(F_m) \sim \frac{2^m}6.
\end{equation}
Therefore
\[
A(F_m) \sim c F_m^\lam, \qquad c = \frac16 \sqrt5^\lam.
\]

Note also that if $F_m \le H < F_{m+1}$ then
\[
\varphi^m (1+o(1)) = F_m \sqrt 5 \le H  \sqrt 5
< F_{m+1}  \sqrt 5 = \varphi^{m+1} (1+o(1)),
\]
and consequently
\[
A(H) < A(F_{m+1}) = \frac{2^{m+1}}6 (1+o(1)) \le \frac13  (H \sqrt 5)^{\lam} (1+o(1))
\]
and
\[
A(H) \ge A(F_m) = \frac{2^m}6 (1+o(1)) \ge \frac1{12} (H \sqrt 5)^\lam (1+o(1)).
\]
These calculations furnish \eqref{OrderOfMagnitude}, in the stronger form
\[
c/2 \le c_1 \le c_2 \le 2c.
\]

\begin{example} By Lemma \ref{KeyRecursion}, as $m \to \infty$ we have
\begin{align*}
A(2F_{m-1}) &= A(2F_{m-1} - F_m) + A(F_{m-1}) - A(0) + 2^{m-3} \\
&= A(F_{m-3}) + A(F_{m-1}) + 2^{m-3} - 1 
\sim \frac{11}{48}2^m \\
&\sim \frac{11}{24} (F_{m-1} \sqrt 5)^{\lam}
= \frac{11 (\sqrt5/2)^\lam}{24}(2F_{m-1})^\lam,
\qquad \frac{11 (\sqrt5/2)^\lam}{24} \approx 0.538,
\end{align*}
and
\begin{align*}
A(F_m+ F_{m-2}) &= A(F_{m-2}) + A(2F_{m-2}) - A(F_{m-4}) + 2^{m-3} \\
&\sim \frac{2^{m-2}}6 + \frac{11}{48}2^{m-1} - \frac{2^{m-4}}6 + 2^{m-3}
\sim \frac{13}{48}2^m \\
&\sim \frac{13}{48}\left( \sqrt 5  \frac{F_m + F_{m-2}}{1 + \varphi^{-2}}\right)^{\lam} 
= \frac{13 \sqrt 5^\lam}{48 (1 + \varphi^{-2})^\lam} (F_m + F_{m-2})^\lam \\
&= \frac{13 \varphi^\lam}{48}  (F_m + F_{m-2})^\lam = \frac{13}{24} (F_m + F_{m-2})^\lam, \qquad \frac{13}{24} \approx 0.542.
\end{align*}
\end{example}

\subsection{An exact formula for the summatory function}

Recall \eqref{fdef}. Applying Lemma \ref{KeyRecursion} several times provides the following more elaborate recursion.

\begin{lemma} \label{SuperRecursion} If $t \ge 2$, $m \ge 2t$, and $F_{m-2t+1} \le x < F_{m-2t + 3}$, then
\[
A(F_m + x) = tA(x) - A(x - F_{m-2t + 2}) + f(t) 2^{m-2t}.
\]
\end{lemma}

\begin{proof} For the base case $t=2$ of our induction, for $m \ge 4$ and
$F_{m-3} \le x < F_{m-1}$ we have
\begin{align*}
A(F_m + x) &= A(x) + A(F_{m-2} + x) - A(x- F_{m-3}) + 2^{m-3} \\
&= 2A(x) - A(x - F_{m-2}) + 2^{m-4} + 2^{m-3} = 2A(x) - A(x - F_{m-2})+ \frac{3}{16} 2^m\\
&= 2A(x) - A(x - F_{m-2})+ f(2) 2^{m-4}.
\end{align*}
Now let $t \ge 3$, and suppose the result holds with $t-1$ in place of $t$. Then for $m \ge 2t$ and $x \in [F_{m-2t+1}, F_{m-2t + 3})$ we have
\begin{align*}
A(F_m + x) &= A(x) + A(F_{m-2} + x) + 2^{m-3} \\
&= t A(x) - A(x - F_{m - 2 - 2(t-1) + 2})
+ \left(1 + \frac{2(4^{t-2}-1)}3 \right) 2^{m-2-2(t-1)} + 2^{m-3} \\
&= tA(x) - A(x - F_{m-2t + 2}) + \left(1 + \frac{2^{2t-3}-2}3 + 2^{2t-3} \right) 2^{m - 2t} \\
&= tA(x) - A(x - F_{m-2t + 2}) + \left(1 + \frac{2(4^{t-1}-1)}3 \right) 2^{m-2t}\\
&= tA(x) - A(x - F_{m-2t + 2}) + f(t) 2^{m-2t}.
\end{align*}
\end{proof}

The following immediate consequence is analogous to Lemma \ref{Robbins}.

\begin{cor} \label{SuperRecursion2} Let $t \ge 2$, $m \ge 2t$, and $F_{m-2t+1} \le x < F_{m-2t + 3}$. Set
\[
(\eps, y) = \begin{cases}
(1, x -  F_{m-2t+2}) &\text{if } F_{m-2t+2} \le x < F_{m-2t + 3} \\
(0,  x -  F_{m-2t+1}) &\text{if } F_{m-2t+1} \le x < F_{m-2t + 2}.
\end{cases}
\]
Then
\[
A(F_m + x) = tA(x) - \eps A(y) + f(t) 2^{m-2t}.
\]
\end{cor}

\bigskip

We now establish \eqref{GeneralRecursion} for $1 \le \ell \le k$. For the base case $\ell = 1$ of our induction, we know from Corollary \ref{SuperRecursion2} that
\[
A(H) = a_1 A(x_1) - \eps_1 a_0 A(x_2) + 
a_0 f(t_1) 2^{m_0 - 2t_1}.
\]
Now suppose that for some $\ell \in \{1,2,\ldots,k-1\}$ we have \eqref{GeneralRecursion}. Then
\begin{align*}
A(H) &= a_\ell (t_{\ell+1} A(x_{\ell+1}) - \eps_{\ell+1} A(x_{\ell+2}) + f(t_{\ell+1}) 2^{m_\ell - 2t_{\ell+1}}) - \eps_\ell a_{\ell-1} A(x_{\ell +1}) \\
&\qquad
+ \sum_{i \le \ell} a_{i-1} f(t_i) 2^{m_{i-1} - 2t_1} \\
&= a_{\ell+1} A(x_{\ell+1}) - \eps_{\ell+1} a_\ell A(x_{\ell+2}) 
+ \sum_{i \le \ell +1} a_{i-1} f(t_i) 2^{m_{i-1} - 2t_i}.
\end{align*}
We have proved \eqref{GeneralRecursion} by induction on $\ell$. 

\bigskip

For $k \ge 1$, inserting \eqref{FibonacciExact} into the $\ell = k$ case of \eqref{GeneralRecursion} yields \eqref{FinalFormula}. Meanwhile, the $k = 0$ case of \eqref{FinalFormula} is precisely \eqref{FibonacciExact}. This completes the proof of Theorem \ref{SuperDuperRecursion}.

\begin{example} Let $m$ be large, and consider
\[
H = F_m + F_{m-7} + F_{m-12} + F_{m-19}.
\]
In this case
\[
t_1 = 4, \quad t_2 = 3, \quad t_3 = 4, \qquad \eps_0 = \eps_1 = \eps_2 = 0.
\]
Therefore
\begin{align*}
A(H) &= 48 \Bigg \lfloor \frac{2^{m-19}}6 + \frac{m-18}2 \Bigg \rfloor
+ f(4) 2^{m-8} + 4 f(3) 2^{m-13} + 12 f(4) 2^{m-20} \\
&\sim (16 + 43 \times 2^{12} + 44 \times 2^7 + 12 \times 43) 2^{m-20} = \frac{45573}{262144} 2^m.
\end{align*}
Thus, as $m \to \infty$, we have
\[
\frac{A(H)}{H^\lam} \to \frac{45573}{262144} 
\left( \frac{\sqrt 5}{1 + \varphi^{-7} + \varphi^{-12} + \varphi^{-19}} \right)^\lam \approx 0.525352.
\]
\end{example}

\section{Subdivision} \label{S4}

In this section, we prove Theorem \ref{AsymptoticBehaviour}. Let $m$ be a large positive integer. We subdivide the discrete interval $[F_m, F_{m+1}) \cap \bZ$ into subintervals 
\[
[p_j, p_{j+1}) \cap \bZ \qquad (0 \le j \le 317810)
\]
according to the initial Zeckendorf digits. The left endpoints $p_0, \ldots, p_{317810}$ are given by
\[
 F_m + \sum_{i \le \ell} F_{m - a_i},
\]
where
\[
\ell \ge 0, \qquad a_1, a_2 - a_1, \ldots, a_\ell - a_{\ell - 1} \ge 2, 
\qquad a_\ell \le 27.
\]
The right endpoints have the same form, except $p_{317811} = F_{m+1}$. Using Theorem \ref{SuperDuperRecursion}, we can show that
\[
A(p_j) \sim v_j 2^m, \qquad p_j \sim w_j \varphi^m
\qquad (0 \le j \le 317811)
\]
as $m \to \infty$, for some computable values of $v_j$ and $w_j$. Then
\[
(1+o(1))L_j \le \frac{A(H)}{H^\lam} \le (1+o(1))U_j 
\qquad (p_j \le H < p_{j+1}),
\]
where
\[
L_j = \frac{v_j}{w_{j+1}^\lam},
\quad
U_j = \frac{v_{j+1}}{w_j^\lam}
\qquad (0 \le j \le 317810).
\]

We carried out these computations using the software \emph{Mathematica} \cite{Wolfram}; the code is provided in Appendix \ref{Acode}. Then
\[
c_1 \ge \min_j L_j, \qquad c_2 \le \max_j U_j.
\]
The software also told us which subintervals attaining the least $L_j$ and the greatest $U_j$, namely
\[
j = 19401, \qquad (a_1,\ldots,a_\ell) = (7,12,18,25)
\]
and
\[
j = 184839, \qquad (a_1,\ldots,a_\ell) = (3,5,8,10,12,16,18,21,23,26),
\]
respectively. Since
\[
\frac{A(p_j)}{p_j^\lam} \sim \frac{v_j}{w_j^\lam},
\]
we thereby also obtained an upper bound for $c_1$ and a lower bound for $c_2$. These calculations delivered \eqref{close}, completing the proof of Theorem \ref{AsymptoticBehaviour}.

\appendix

\section{Code for $R(H)$} \label{Rcode}

Lines 2--7 are \emph{Rosetta code} \cite{Rosetta}, available for general use under the GNU Free Documentation License, version 1.2. The value of $H$ in the first line can be changed. \\

\begin{Verbatim} [fontsize=\scriptsize]
H = 1234;
zeckendorf[0] = 0;
zeckendorf[n_Integer] := 
  10^(# - 1) + zeckendorf[n - Fibonacci[# + 1]] &@
   LengthWhile[
    Fibonacci /@ 
     Range[2, Ceiling@Log[GoldenRatio, n Sqrt@5]], # <= n &];
Z = IntegerDigits[zeckendorf[H]];
l = Total[Z];
X = ConstantArray[0, l];
t = 1;
If[l == 1, Floor[(Length[Z] + 1)/2],
 For[i = 1, i < Length[Z] + 1, i++,
  If[Z[[i]] == 1, X[[t]] = Length[Z] - i + 2; t++;,]
  ];
 T = ConstantArray[0, l - 1];
 Ep = ConstantArray[0, l - 1];
 For[i = 1, i < l, i++,
  T[[i]] = Floor[(X[[i]] - X[[i + 1]] + 2)/2];
  Ep[[i]] = 2 T[[i]] - 1 - X[[i]] + X[[i + 1]];
  ];
 a = ConstantArray[1, l];
 a[[2]] = T[[1]];
 For[i = 3, i < l + 1, i++,
  a[[i]] = T[[i - 1]] a[[i - 1]] - Ep[[i - 2]] a[[i - 2]]
  ];
 a[[l]]*Floor[X[[l]]/2] - a[[l - 1]]*Ep[[l - 1]]
 ]
\end{Verbatim}

\section{Code for $A(H)$} \label{Acode}

\begin{Verbatim} [fontsize=\scriptsize]
P = (1 + Sqrt[5])/2;
L = Log[2]/Log[P];
l = 27;
X = ConstantArray[0, {Fibonacci[l + 1], Floor[l/2]}];
t = 1;
X[[2, 1]] = l;
For[i = 3, i < Fibonacci[l + 1] + 1, i++,
 If[X[[i - 1, t]] == l || X[[i - 1, t]] == l - 1,
  If[(t > 1 && X[[i - 1, t]] - X[[i - 1, t - 1]] == 2),
   t--;
   For[j = t, j > 0, j--,
    If[j == 1, t = j;
     For[k = 1, k < t, k++, X[[i, k]] = X[[i - 1, k]]];
     X[[i, t]] = X[[i - 1, t]] - 1; j = 0, 
     If[X[[i - 1, j]] - X[[i - 1, j - 1]] != 2, t = j;
      For[k = 1, k < t, k++, X[[i, k]] = X[[i - 1, k]]];
      X[[i, t]] = X[[i - 1, t]] - 1; j = 0]
     ]
    ]
   ,
   X[[i]] = X[[i - 1]];
   X[[i, t]]--;
   ],
  t++;
  X[[i]] = X[[i - 1]];
  X[[i, t]] = l;
  ]
 ]
T = ConstantArray[0, {Fibonacci[l + 1], Floor[l/2]}];
For[i = 2, i < Fibonacci[l + 1] + 1, i++,
 T[[i, 1]] = Floor[(X[[i, 1]] + 2)/2]
 ]
For[i = 2, i < Fibonacci[l + 1] + 1, i++,
 For[j = 2, j < Floor[l/2] + 1, j++,
  If[X[[i, j]] == 0, , 
   T[[i, j]] = Floor[(X[[i, j]] - X[[i, j - 1]] + 2)/2]
   ]
  ]
 ]
Ep = ConstantArray[0, {Fibonacci[l + 1], Floor[l/2]}];
For[i = 2, i < Fibonacci[l + 1] + 1, i++,
 Ep[[i, 1]] = 2 T[[i, 1]] - 1 - X[[i, 1]]
 ]
For[i = 2, i < Fibonacci[l + 1] + 1, i++,
 For[j = 2, j < Floor[l/2] + 1, j++,
  If[X[[i, j]] == 0, , 
   Ep[[i, j]] = 2 T[[i, j]] - 1 - X[[i, j]] + X[[i, j - 1]]
   ]
  ]
 ]
a = ConstantArray[0, {Fibonacci[l + 1], Floor[l/2] + 1}];
For[i = 1, i < Fibonacci[l + 1] + 1, i++, a[[i, 1]] = 1];
For[i = 2, i < Fibonacci[l + 1] + 1, i++, a[[i, 2]] = T[[i, 1]]];
For[i = 2, i < Fibonacci[l + 1] + 1, i++,
 For[j = 3, j < Floor[l/2] + 2, j++,
  If[X[[i, j - 1]] == 0, , 
   a[[i, j]] = 
    T[[i, j - 1]] a[[i, j - 1]] - Ep[[i, j - 2]] a[[i, j - 2]]
   ]
  ]
 ]
f = Function[t, 1 + (2/3) (4^(t - 1) - 1)];  
k = ConstantArray[0, {Fibonacci[l + 1] + 1}];
k[[1]] = (1/6);
k[[Fibonacci[l + 1] + 1]] = (1/3);
For[i = 2, i < Fibonacci[l + 1] + 1, i++,
 For[j = Floor[l/2], j > 0, j--,
  If[X[[i, j]] == 0, ,
   k[[i]] = (a[[i, j + 1]]/(6*2^(X[[i, j]]))) + 
     Sum[(a[[i, l]] f[T[[i, l]]]/2^(X[[i, l - 1]] + 2 T[[i, l]])), {l,
        2, j}] + (a[[i, 1]] f[T[[i, 1]]]/2^(2 T[[i, 1]]));
   j = 0
   ]
  ]
 ]
p = ConstantArray[0, {Fibonacci[l + 1] + 1}];
p[[1]] = 1; p[[Fibonacci[l + 1] + 1]] = P;
For[i = 2, i < Fibonacci[l + 1] + 1, i++,
 For[j = Floor[l/2], j > 0, j--,
  If[X[[i, j]] == 0, ,
   p[[i]] = 1 + Sum[P^(-X[[i, k]]), {k, 1, j}];
   j = 0
   ]
  ]
 ]
LU = ConstantArray[0, {Fibonacci[l + 1], 2}];
For[i = 1, i < Fibonacci[l + 1] + 1, i++,
 LU[[i, 1]] = k[[i]]*(Sqrt[5]/p[[i + 1]])^L;
 LU[[i, 2]] = k[[i + 1]]*(Sqrt[5]/p[[i]])^L
 ] 
NumberForm[N[Min[LU]], 8]
NumberForm[N[Max[LU]], 8]
Position[LU, Min[LU]]
Position[LU, Max[LU]]

X[[19401]]
X[[184839]]
NumberForm[N[k[[19401]]*(Sqrt[5]/p[[19402]])^L], 8]
NumberForm[N[k[[19401]]*(Sqrt[5]/p[[19401]])^L], 8]
NumberForm[N[k[[184840]]*(Sqrt[5]/p[[184839]])^L], 8]
NumberForm[N[k[[184839]]*(Sqrt[5]/p[[184839]])^L], 8]
\end{Verbatim}
 
\providecommand{\bysame}{\leavevmode\hbox to3em{\hrulefill}\thinspace}

\end{document}